\documentclass[reqno,A4paper]{amsart}

\usepackage{amsmath}
\usepackage{amssymb}
\usepackage{amsthm}
\usepackage{enumerate}
\usepackage{graphicx}
\usepackage{mathrsfs} 
\usepackage{eqlist}
\usepackage{array}

\theoremstyle{plain}
\newtheorem{theorem}{Theorem}[section]

\newtheorem{lemma}{Lemma}[section]

 \newtheorem{corollary}[theorem]{Corollary}

 \newtheorem{proposition}[theorem]{Proposition}

\theoremstyle{definition}
\newtheorem{definition}{Definition}

\newtheorem*{acknowledgement}{\textup{Acknowledgement}}

\numberwithin{equation}{section}

\newcommand{\aaa}{\alpha}
\newcommand{\bbb}{\beta}
\newcommand{\eps}{\varepsilon}

\newcommand{\ess}{\emptyset}
\newcommand{\oo}{\infty}
\newcommand{\sm}{\setminus}
\newcommand{\sse}{\subset}
\newcommand{\N}{\ensuremath{\mathbb N}} 
\newcommand{\R}{\ensuremath{\mathbb R}} 
\newcommand{\Z}{\ensuremath{\mathbb Z}} 

\DeclareMathOperator{\lip}{lip\kern-0.8pt}
\DeclareMathOperator{\Lip}{Lip\kern-0.8pt}
\newcommand*{\Union}{\bigcup}

\begin{document}

\title[Lipschitz one sets modulo sets of measure zero]
{Lipschitz one sets modulo sets of measure zero}
\author[Zolt\'an Buczolich, Bruce Hanson, Bal\'azs Maga \and G\'asp\'ar V\'ertesy]%
{Zolt\'an Buczolich*, Bruce Hanson**, Bal\'azs Maga*** \and G\'asp\'ar V\'ertesy****}

\newcommand{\acr}{\newline\indent}

\address{\llap{*\,}Department of Analysis\acr
                   ELTE E\"otv\"os Lor\'and University\acr
                   P\'azm\'any P\'eter S\'et\'any 1/c\acr
                   1117 Budapest\acr
                   HUNGARY\acr
                   \acr
                   ORCID Id: 0000-0001-5481-8797}
\urladdr{http://buczo.web.elte.hu}

\email{buczo@caesar.elte.hu}

\address{\llap{**\,}Department of Mathematics, Statistics and Computer Science\acr
                    St.\ Olaf College\acr
                    Northfield, Minnesota 55057\acr
                    USA}
\email{hansonb@stolaf.edu}

\address{\llap{***\,}Department of Analysis\acr
                   ELTE E\"otv\"os Lor\'and University\acr
                   P\'azm\'any P\'eter S\'et\'any 1/c\acr
                   1117 Budapest\acr
                   HUNGARY}
\email{magab@caesar.elte.hu}

\address{\llap{****\,}Department of Analysis\acr
                   ELTE E\"otv\"os Lor\'and University\acr
                   P\'azm\'any P\'eter S\'et\'any 1/c\acr
                   1117 Budapest\acr
                   HUNGARY}
\email{vertesy.gaspar@gmail.com}


\thanks{Zolt\'an Buczolich was supported by the Hungarian National Research, Development and Innovation Office--NKFIH, Grant 124003.}
\thanks{Bal\'azs Maga was initially supported by the ÚNKP-18-2 New National Excellence of the Hungarian Ministry of Human Capacities, later on by the ÚNKP-19-3 New National Excellence Program of the Ministry for Innovation and Technology, and during the entire period by the Hungarian National Research, Development and Innovation Office–NKFIH, Grant 124749.}
\thanks{G\'asp\'ar V\'ertesy was supported by the \'UNKP-18-3 New National Excellence Program of the Ministry of Human Capacities, and by the Hungarian National Research, Development and Innovation Office–NKFIH, Grant 124749.}


\begin{abstract}
We denote the local ``little" and ``big" Lipschitz functions of a function
$f: \R\to \R$ by $ \lip f$ and $ \Lip f$. In this paper we continue our research concerning the following question. Given a set $E {\subset}  \R$ is it possible to find a continuous function $f$ such that
$ \lip f=\mathbf{1}_E$ or $ \Lip f=\mathbf{1}_E$?

In giving some partial answers to this question uniform density type (UDT) and strong uniform density type (SUDT) sets play an important role. 

In this paper we show that modulo sets of zero Lebesgue measure any measurable set coincides with a $\Lip 1$ set. 

On the other hand, we prove that there exists a measurable SUDT set $E$  such that for any $G_\delta$ set $\widetilde{E}$ satisfying $|E\Delta\widetilde{E}|=0$ the set $\widetilde{E}$ does not have UDT. Combining these two results we obtain that there exists $\Lip 1$ sets not having UDT, that is, the converse of one of our earlier results does not hold.
\end{abstract}

\maketitle

\section{Introduction}\label{*secintro}
If $f:\R \to \R$ is continuous, then the so-called ``big Lip'' and ``little lip'' functions are defined as follows:

\begin{equation}
 \Lip
  f(x)=\limsup_{r\to 0^+}M_f(x,r),\qquad\label{Lipdef}
 \lip
  f(x)=\liminf_{r\rightarrow 0^+}M_f(x,r),
 \end{equation}
where
$$M_f(x,r)=\frac{\sup\{|f(x)-f(y)| \colon |x-y| \le r\}}r.$$

By the Rademacher-Stepanov Theorem  \cite{MaZa}
if $\Lip f(x)<\infty $ for Lebesgue almost every $x $, then $f$ is differentiable almost everywhere.  On the other hand, in \cite{BaloghCsornyei} Balogh and Cs\"{o}rnyei  showed that this property is not true if one replaces $\Lip f$ with  $\lip f$.
This line of research was continued in \cite{Hanson2} and  \cite{BHRZ}.

As other activity concerning $\lip $ exponents it is also worth mentioning the very recent result \cite{MaZi}.

The current paper is a continuation of \cite{[BHMVlip]}.   (At the time of acceptance of this paper \cite{[BHMVlip]} was not accepted/published hence our references to numbered Theorems/Lemmas etc. in \cite{[BHMVlip]} are to the first
arXiv preprint version of \cite{[BHMVlip]}.)

Following \cite{[BHMVlip]}, we say that $E\subset \R$ is $\Lip 1$ ($\lip 1$) if there exists a continuous function $f$ defined on $\R$ so that
$\Lip f=\mathbf{1}_E$ ($\lip f=\mathbf{1}_E$).  In \cite{[BHMVlip]} we considered the challenging problem of characterizing these sets, focusing primarily on the $\Lip 1$ case.  According to \cite[Theorem 4.1]{[BHMVlip]}, being a $G_\delta$ set is a necessary, but not sufficient condition for being a $\Lip 1$ set.  

Our sufficient conditions of sets being $\Lip 1$ rely on assumptions about uniform density properties of these sets. First we need to define the sets
$E^{\gamma,\delta}$, as we did in \cite[Definition 1.1]{[BHMVlip]}:
\begin{definition}\label{*defEgd}
Let $E\subseteq\mathbb{R}$ be measurable and $\gamma,\delta>0$.
Then
\begin{displaymath} E^{\gamma,\delta}=\left\{x\in\mathbb{R} \bigm| \forall r\in (0,\delta],\texttt{ }\max\left\{\frac{|(x-r,x)\cap E|}{r},\frac{|(x,x+r)\cap E|}{r}\right\}\geq\gamma \right\}.
\end{displaymath}
\end{definition}
(Note that we use $|A|$ to indicate the Lebesgue measure of a set $A$.)

In \cite[Definitions 1.1 and 5.3]{[BHMVlip]} the following density conditions were introduced:


\begin{definition}\label{*defudt}
We say that $E$ has uniform density type (UDT) if there exist sequences $\gamma_n\nearrow 1$ and $\delta_n\searrow 0$ such that $E\subseteq\bigcap_{k=1}^{\infty}\bigcup_{n=k}^{\infty}E^{\gamma_n,\delta_n}$.

On the other hand, $E$ has strong uniform density type (SUDT) if there exist sequences $\gamma_n\nearrow 1$ and $\delta_n\searrow 0$ such that $E\subseteq\bigcup_{k=1}^{\infty}\bigcap_{n=k}^{\infty}E^{\gamma_n,\delta_n}$.
\end{definition}

One of the main results from \cite[Theorem 5.5]{[BHMVlip]}, states that if a set $E$ is  $G_\delta$ and UDT, then $E$ is $\Lip 1$.

In the present paper we show that every measurable subset of $\R$ is ``close'' to being a $\Lip 1$ set.  More precisely, we prove

\begin{theorem}\label{*EEtappr}
For every measurable set $E$ there exists a $G_{\delta}$, $\Lip 1$ set $\widetilde{E}$ such that $|E\triangle \widetilde{E}| = 0$.
\end{theorem}
In measure theory such theorems are often not too difficult, but in our case the proof of this theorem is not that easy. 

 On the other hand, we also prove the following:

\begin{theorem}\label{*thSUDTnonUDT}
There exists an $F_\sigma$ set $E\subseteq{\mathbb{R}}$ having SUDT such that for any $G_\delta$ set $\widetilde{E}$ satisfying $|E\Delta\widetilde{E}|=0$ the set $\widetilde{E}$ does not have UDT.
\end{theorem}
 Combining these two theorems yields $\Lip 1$ sets which fail to be UDT so the converse of \cite[Theorem 5.5]{[BHMVlip]} is false.

The layout of this paper is as follows:  In Section \ref{*secprel} we introduce our notation and recall some of the results from \cite{[BHMVlip]}.  In Section \ref{secsubsuper}, we introduce a class of Cantor sets which have SUDT and use them to construct the set $E$ given in Theorem \ref{*thSUDTnonUDT}.   Finally, Section \ref{*seclip15} is devoted to the proof of Theorem \ref{*EEtappr}.

\section{Notation and preliminaries}\label{*secprel}

We recall several definitions and results, from \cite{[BHMVlip]}.

\begin{definition}\label{l converge}
We write $I_n \stackrel{l}{\to} x$ (resp.~$I_n \stackrel{r} \to x$) if $(I_n)$ is a sequence of closed intervals with $I_n=[x-r_n,x]$ (resp.~$I_n=[x,x+r_n]$) and $r_n \searrow 0$.
\end{definition}

\begin{definition}\label{one side dense}
The set $E$ is  \emph{right (left) dense} at $x$ if for any sequence $(I_n)$ such that
$I_n \stackrel{r}{\to} x$ ($I_n \stackrel{l}{\to}x$) we have $\frac{|E \cap I_n|}{|I_n|}\to 1$. In this case, we say that $x$ is a \emph{right density point} (\emph{left density point}) of $E$.
The set $E$ is \emph{one-sided dense} if $E$ is either right or left dense at every point $x \in E$.
\end{definition}


From Definition \ref{*defEgd} it is straightforward to check the following lemma from \cite[Lemma 5.1]{[BHMVlip]}:

\begin{lemma}\label{*lemegdcloa} For any $\gamma,\delta>0$ the set $E^{\gamma,\delta}$ is closed. \end{lemma} \label{closed}

The following proposition and theorem were also proved in \cite[Proposition 5.4 and Theorem 5.5]{[BHMVlip]}:

\begin{proposition} \label{*prop06}
Let $E, E_1, E_2, ...$ be measurable subsets of $\mathbb{R}$.
\begin{enumerate}
\item If a set $E$ has SUDT then it also has UDT.
\item Any interval has SUDT (and hence UDT).
\item If $E_1, E_2, ...$ have UDT (resp.
SUDT) then $E=\bigcup_{n=1}^{\infty}E_n$ also has UDT (resp.
SUDT).
\item There exists $E$ which has SUDT but its closure $\overline{E}$ does not have UDT.
\end{enumerate}
\end{proposition}

\begin{theorem}
\label{*thUDTLip1}
 Assume that $E$ is $G_\delta$ and $E$ has UDT. Then there exists a continuous function $f$ satisfying $\Lip f =\mathbf{1}_E$, that is the set $E$ is $\Lip 1$.
\end{theorem}

The proofs of (i) and (ii) in the proposition are quite elementary, while the other two parts  are nice exercises and we encourage the reader to consider them as such. However, the proof of the theorem is quite elaborate as one of the main results of \cite{[BHMVlip]}.

\section{An SUDT set which is not approximable by a $G_\delta$ UDT set}
\label{secsubsuper}

\begin{definition}\label{*notcantsim}
Suppose that $(\alpha_n)$ satisfies $0 < \alpha_n < 1$ for all $n \in  \mathbb{N}
 $ and $E$ is a Cantor set constructed by starting with $[0,1]$ and then removing the open interval of length $\alpha_1$ centered at 1/2 from $[0,1]$.
Then continuing with a standard ``middle interval'' construction after the $n$th step there will be $2^n$ closed intervals remaining, each of the same length.
If $I$ is one of these intervals at the next stage of the construction we remove from $I$ an open interval centered at the midpoint of $I$ and of length $\alpha_{n+1} |I|$.
We let
$\mathcal{I}_n$ be the collection of closed intervals remaining after the $n$th step of the construction. For arbitrary $n$, the length $d_n$ of each of these intervals can be obtained recursively by letting  $d_0=1$ and $2d_{n}=(1-\alpha_n)d_{n-1}$.
Finally, we define $E=\bigcap_{n\in\mathbb{N}} E_n$ where $E_n = \bigcup_{I \in \mathcal{I}_n}I$.
In this case we use the notation $E \sim (\alpha_n)$.
\end{definition}

\begin{theorem}\label{fat cantor density}
Using  Definition \ref{*notcantsim} suppose that
 $E \sim (\alpha_n)$ where $\sum \alpha_n < \infty$.
Then $E$ is a nowhere dense closed set, which has SUDT.
\end{theorem}

\begin{proof}
Let $n\in\mathbb{N}$ be arbitrary. Suppose that $I = [a,b] \in \mathcal{I}_n$, so $|I|=d_n$.
Note that
$$\frac{|E\cap I|}{|I|}=\prod_{k=n+1}^\infty (1-\alpha_k)=\beta_n,$$
where $\beta_n \nearrow 1$ as $n\to\infty$ since $\sum \alpha_n < \infty$.
Choose $\gamma_n=1-12(1-\beta_n)$ and $\delta_n=\frac12 d_n$.

We claim that
\begin{equation}\label{E has SUDT}
E \subset \cup_{k=1}^\infty \cap_{n=k}^\infty E^{\gamma_n,\delta_n},
\end{equation}
and therefore $E$ has SUDT.

To verify the claim let $x \in E$ and for each $n \in  \mathbb{N}
 $ choose $I_n=[a_n,b_n]\in \mathcal{I}_n$ such that $x \in I_n$ so $\{x\}=\bigcap I_n$.
Now let $r_n=\max\{x-a_n,b_n-x\}$.
We assume without loss of generality that
$\alpha_n < 1/3$ for each $n$.
Then it follows easily that
$$d_{n+1}> \frac13 d_n, \quad \quad \frac12d_n \le r_n \le d_n, \quad \quad \frac16 r_n < r_{n+1} < r_n.$$
For each $n \in  \mathbb{N}
 $ we let $J_n=[x-r_n,x]=[a_n,x]$ if $x-a_n > b_n-x$ and $J_n =[x,x+r_n]=[x,b_n]$ otherwise.
Then it follows from $r_n \ge \frac12d_n$ that $\frac{|E \cap J_n|}{|J_n|} \ge 1-2(1-\beta_n)$.
Similarly, for every $\delta$ satisfying $r_{n+1} \le \delta \le r_n$ we can take $J=J_{n,\delta}$ to be a closed interval of length $\delta$ with $x$ as an endpoint and contained in $J_n$. 
Since $$ |J\sm E| \leq  |J_n \sm E| \leq  |I_n \sm E| = (1 -\beta_n )|I_n | $$  $$= (1 -\beta_n )d_n \leq  12(1 -\beta_n )r_{n+1} \leq  12(1 -\beta_n )|J|.
$$
we have $\frac{|E \cap J|}{|J|} \ge 1- 12 (1-\beta_n)=\gamma_n$.
It now follows easily that $x \in \bigcup_{k=1}^\infty \bigcap_{n=k}^\infty E^{\gamma_n,\delta_n}$ and therefore (\ref{E has SUDT}) holds.
\end{proof}

At a first glance, one might believe that if $K$ is an SUDT set, then each of its points is a left or right density point. We will refute this belief by proving that the SUDT set provided by Theorem \ref{fat cantor density} does not have this property. In order to state this result a bit more generally, we introduce weakly nowhere dense sets:

\begin{definition}\label{*defwnd}
The set $E\subseteq{\mathbb{R}}$ is \emph{weakly nowhere dense}  if for any interval $J$, the subset $E\cap{J}$ does not have full measure in $J$.
\end{definition}

Let us notice that if $E$ is weakly nowhere dense and $\alpha\in (0,1)$ is a fixed positive real number then
 by Lebesgue's density theorem applied to the complement of $E$,
for any interval $J$ there exists a subinterval $I\subseteq{J}$ such that $|I\cap{E}|<\alpha|I|$. Moreover, it is also clear that a nowhere dense set is weakly nowhere dense.

\begin{theorem}\label{*th010}
Assume that $E$ is a weakly nowhere dense set.
Then the set $D_L(E)$ of left density points (resp. the set $D_R(E)$ of right density points) is of first category in $E$.
\end{theorem}

\begin{proof}[Proof of Theorem \ref{*th010}]
We use an argument similar to the one used in \cite{[Bu1]}.
Proceeding towards {a} contradiction, assume that $D_L(E)$ is of second category in $E$.
Set
\begin{equation}\label{*Hndef}
H_n=\left\{x\in{E}:\forall h\in\Big (0,\frac{1}{n}\Big ),\  \frac{|(x-h,x)\cap{E}|}{h}>{\frac{1}{2}}\right\}.
\end{equation}
Then $D_L(E)\subseteq\bigcup_{n=1}^{\infty}H_n$ clearly holds, hence there exists $n$ such that $H_n$ is of second category in $E$.
Consequently, there exists an open interval $J$ such that $J\cap{E}\neq\emptyset$ and $H_n$ is dense in $J\cap{E}$.
As $E$ is weakly nowhere dense, by the previous remark we can choose an interval $I=(a,b)$ such that $[a,b]\subseteq{J}$, we have $b-a<\frac{1}{n}$,  $\frac{|(a,b)\cap{E}|}{b-a}<{\frac{1}{4}}$,
and $E \cap J \cap [b,\infty)  \neq \emptyset$.
Moreover, {we may assume $b \in E$} as otherwise we can translate the interval $I$ {to the right} until we arrive at such a point.
Now since $H_n$ is dense in $J\cap{E}$ we can choose a point $x\in H_n$ such that $|x-b|<{(b-a)}/{4}$.
However, for this $x$ and $h=b-a<\frac{1}{n}$ we have
\begin{displaymath}
\frac{|(x-h,x)\cap{E}|}{h}\leq\frac{\frac{1}{4}(b-a)+\frac{1}{4}(b-a)}{(b-a)}=\frac{1}{2},
\end{displaymath}
contradicting $x\in{H_n}$. {This} concludes the proof.
\end{proof}

\begin{corollary}\label{noonesideddense}
If $E$ is a non-empty weakly nowhere dense $G_\delta$ set, then it has points which are not one-sided density points.
\end{corollary}

Notably, the nowhere dense, closed SUDT set provided by Theorem \ref{fat cantor density} has points which are not one-sided density points.

\begin{proof}[Proof of Corollary \ref{noonesideddense}]
The set of one-sided density points is the union of $D_L(E)$ and $D_R(E)$, hence it is a first category set by the previous theorem.
However, as $E\subseteq{\mathbb{R}}$ is $G_\delta$, it is a Baire space by Alexandrov's Theorem (see \cite{Ku} for example), thus we can apply {the} Baire Category Theorem to obtain the statement of the corollary.
\end{proof}

Now we will prove Theorem \ref{*thSUDTnonUDT} with the help of Theorem \ref{fat cantor density}.

\begin{proof}[Proof of Theorem  \ref{*thSUDTnonUDT}.]
Using  Definition \ref{*notcantsim}
let $E_n^*\sim(\alpha_{k,n})_{k=1}^\infty$ such that $|E_n^*|=\frac{1}{2^n}$.  {It} is easy to check that there exist such sequences $(\alpha_{k,n})_{k=1}^\infty$ satisfying $\sum_{k=1}^{\infty}\alpha_{k,n}<\infty$.
Then the set of intervals which are contiguous to any of these sets is countable.
Now set $E_1=E_1^*$.
Next we let $E_2$ be a homothetic image of $E_2^*$ centered in a contiguous interval to $E_1$ in $[0,1]$.
Secondly, we define $E_3$ as a homothetic image of $E_3^*$ centered in a contiguous interval to $E_1\cup E_2$, etc.
We proceed recursively so that none of the occuring complementary intervals remain empty by the end {of} the process.
By countability we can do so.
Consequently the set $E=\bigcup_{n=1}^{\infty}E_n$ is a dense, $F_\sigma$ set.
By Theorem \ref{fat cantor density}
and (iii) of Proposition \ref{*prop06}
 it has SUDT. We claim that it is a good example for the statement of the theorem.

To verify that take any $G_\delta$ set $\widetilde{E}$ satisfying $|E\Delta\widetilde{E}|=0$.
By construction, the set $E$ has positive measure in any nontrivial subinterval of $[0,1]$.
Consequently $\widetilde{E}$ must be dense in $[0,1]$.
As {$\widetilde{E}$} is also $G_\delta$, we have that $\widetilde{E}$ is residual.
Proceeding towards a contradiction, assume that $\widetilde{E}$ has UDT, that is
\begin{equation}\label{*eq1}
\widetilde{E}\subseteq\bigcap_{i=1}^{\infty}\bigcup_{j=i}^{\infty}\widetilde{E}^{\gamma_j,\delta_j}
\end{equation}
for suitable sequences $(\gamma_j)$, $(\delta_j)$.
As $\widetilde{E}$ equals $E$ modulo null-sets, we obviously have that $E^{\gamma,\delta}=\widetilde{E}^{\gamma,\delta}$ for any choice of $\gamma,\delta$.
Hence \eqref{*eq1} can be rewritten as
\begin{equation}
\widetilde{E}\subseteq\bigcap_{i=1}^{\infty}\bigcup_{j=i}^{\infty}E^{\gamma_j,\delta_j}.
\end{equation}
In particular, we have
\begin{equation}
\widetilde{E}\subseteq\bigcup_{j=1}^{\infty}E^{\gamma_j,\delta_j}.
\end{equation}
By Lemma \ref{*lemegdcloa}	each of the sets $E^{\gamma_j,\delta_j}$ is closed and their union contains the residual set $\widetilde{E}$.
Consequently, for suitable $i$ the set $E^{\gamma_i,\delta_i}$ contains an open interval $I$.
The definition of $E^{\gamma_i,\delta_i}$ implies that $I$ cannot contain any
density points of the complement of $E$. 
Hence $E$ is of full measure in $I$.
However, we will show below that $E$ is weakly nowhere dense and hence
cannot be of full measure in $I$ which is a contradiction proving the theorem.

Therefore, suppose that $I$ is a given non-empty open subinterval of $[0,1]$. 
Since $E$ is dense in $[0,1]\supset I$
choose $n>1$  such that we can find $a,b\in I\cap \cup_{l=1}^{n}E_{l}$, $a<b$
 such that $(a,b)\cap \cup_{l=1}^{n}E_{l}=\ess$. Then by the definition of $E$
 we have
 $$\left|E\cap (a,b)\right|\leq\sum_{j=n}^{\infty}\frac{1}{2^j}\frac{b-a}{2}<\frac{1}{2}(b-a).$$
\end{proof}

\section{Approximating measurable sets with $\Lip 1$ sets}\label{*seclip15}

\begin{lemma}\label{suru}
If $U\subset\R$ is open, $\widetilde{H}\subset U$ is measurable and $\eps>0$, then there is an open set $H\subset U$ such that $|\widetilde{H}\setminus H|=0$, and if $I=(a,b)$ is a bounded component of $H$, then $\widetilde{H}$ is right dense at $a$ and left dense at $b$ and for every $r\in(0,b-a)$ we have
\begin{equation}\label{H_jo}
\max\Big\{\frac{|(a,a+r)\setminus\widetilde{H}|}{r}, \frac{|(b-r,b)\setminus\widetilde{H}|}{r}\Big\} < \eps.
\end{equation}
\end{lemma}

\begin{proof}
If $|\widetilde{H}|=0$ then $H:=\emptyset$ is a suitable choice, hence we can assume that $|\widetilde{H}|>0$.

First we prove that if $x$ is a density point of $\widetilde{H}$ and $\eps_x>0$, then there is an interval $I_x=(a_x,b_x)\subset U$ which contains $x$, its endpoints are density points of $\widetilde{H}$ and for every $r_x\in(0,b_x-a_x)$ we have
\begin{equation}\label{dense2}
\max\Big\{\frac{|(a_x,a_x+r_x)\setminus\widetilde{H}|}{r_x},\frac{|(b_x-r_x,b_x)\setminus\widetilde{H}|}{r_x}\Big\} < \eps_x.
\end{equation}
Since $x$ is a density point of $\widetilde{H}$ we can take an open interval $I'_x=(a'_x,b'_x)\subset U$ centered at $x$ for which
\begin{equation}\label{dense1}
\frac{|I'_x\setminus\widetilde{H}|}{|I'_x|} < \frac{\eps_x}{16}.
\end{equation}
Let 
\begin{equation}\label{H_x}
\begin{split}
H_x := \Big\{ x'\in [a'_x,x] : &\text{ $\exists$ $r'_{x'}\in \Big[0,\frac{b'_x-a'_x}{2}\Big]$ such that } \\
&\text{ $\dfrac{|[x',x'+r'_{x'}]\setminus \widetilde{H}|}{r'_{x'}}\ge\eps_x$} \Big\}.
\end{split}
\end{equation}

For every $x'\in H_x$ fix such an $r'_{x'}$.
Recall the well-known fact that from a finite covering by intervals we can extract a subcover which 
 covers the same set, but which does not cover any point more than twice.
Choose a finite subset $X'$ of $H_x$ such that
\begin{equation}\label{max_ketto}
\text{for every $z\in\R$  we have that }\#\Big\{x'\in X' : z\in [x',x'+r'_{x'}]\Big\} \le 2
\end{equation}
and
\begin{equation}\label{kozelites}
\Big |\Union_{x'\in H_x} [x',x'+r'_{x'}]  \setminus \Union_{x'\in X'} [x',x'+r'_{x'}]\Big |< \frac{|I'_x|}{16}.
\end{equation}

By \eqref{kozelites} we obtain
\begin{align*}
|H_x| &\le \frac{|I'_x|}{16}+\Big |\Union_{x'\in X'} [x',x'+r'_{x'}]\Big |
\le  \frac{|I'_x|}{16}+\sum_{x'\in X'} r'_{x'} \\
&\underset{\text{by \eqref{H_x} }}{\le} \frac{|I'_x|}{16}+\sum_{x'\in X'}\frac{1}{\eps_x} |[x',x'+r'_{x'}]\setminus\widetilde{H}|\\
&\underset{\text{by \eqref{max_ketto}}}{\le} \frac{|I'_x|}{16}+\frac{1}{\eps_x}\cdot 2 |I'_x\setminus \widetilde{H} | \\
&\underset{\text{by \eqref{dense1} }}{\le} \frac{|I'_x|}{16}+\frac{1}{\eps_x}\cdot 2 \cdot \frac{\eps_x}{16} |I'_x| = \frac{3|I'_x|}{16}=\frac{3(b_{x}'-a_{x}')}{16}.
\end{align*}
Thus, by Lebesgue's density theorem, there exists a density point $a_x$ of $\widetilde{H}$ in $\left((3a'_x+b'_x)/4,(a'_x+b'_x)/2\right)$ such that
$a_{x}\not\in H_{x}$ and hence
 $r_x\in \left(0,\frac{b'_x-a'_x}{2}\right)$ implies
$$
\frac{|(a_x,a_x+r_x)\setminus\widetilde{H}|}{r_x} < \eps_x.
$$
Similarly, there exists a density point $b_x$ of $\widetilde{H}$ in $((a'_x+b'_x)/2,(a'_x+3b'_x)/4)$ such that $r_x\in (0,\frac{b'_x-a'_x}{2})$ implies
$$
\frac{|(b_x-r_x,b_x)\setminus\widetilde{H}|}{r_x} < \eps_x.
$$
As $b_x-a_x<\frac{b'_x-a'_x}{2}$ and $x=\frac{a'_x+b'_x}{2}\in (a_x,b_x)$, the points $a_x$ and $b_x$ satisfy \eqref{dense2}.

We choose a subset $X=\{x_1,x_2,\ldots\}$ of the density points of $\widetilde{H}$ with their corresponding neighbourhoods $\left\{I_{x_1},I_{x_2},\ldots\right\}$, a sequence of positive numbers $\{\eps_{x_1},\eps_{x_2},\ldots\}$ and a strictly increasing sequence of natural numbers $(m_i)_{i=1}^\infty$ such that they satisfy properties (\ref{*dden}-\ref{nagy}):
\begin{equation}\label{*dden}
\text{\eqref{dense2} holds  with $x$ replaced with $x_n$,}
\end{equation}
 \begin{equation}\label{szoveg_lemma}
 \text{each real number is contained in at most two of $\{I_{x_1},I_{x_2},\ldots\}$,}
\end{equation}
\begin{equation}\label{*epskk}
\eps_{x_k} < \frac{\eps}{4} \text{ for every $k\in\N$},
\end{equation} 
\begin{equation}\label{eps'_kicsi}
\eps_{x_k} < \frac{\eps}{4i} \text{ for every $i\in\N$ and $k \ge m_i$},
\end{equation}
and
\begin{equation}\label{nagy}
\Big |\left(\widetilde{H}\cap [-i,i]\right) \setminus \Union_{j=1}^{m_i} I_{x_j}\Big | < \frac{1}{i}.
\end{equation}
Denote
\begin{equation}\label{H}
H := \Union_{j=1}^\infty I_{x_j}.
\end{equation}
By \eqref{nagy} we have that $
|\widetilde{H}\setminus H|=0.
$
 Let $I=(a,b)$ be a component of $H$ and $r\in (0,b-a)$.
By \eqref{H}, there is a $j_r\in\N$ such that
\begin{equation}\label{nagy_resz}
\Big |(a,a+r)\setminus\Union_{j=1}^{j_r} I_{x_j}\Big | < r\frac{\eps}{2}.
\end{equation}
Hence we have
\begin{align*}
&\frac{|(a,a+r)\setminus\widetilde{H}|}{r}  \le  \frac{\Big |(a,a+r)\setminus \Union\limits_{j=1}^{j_r} I_{x_j}\Big |+\Big |\Big (\Union\limits_{j=1}^{j_r} I_{x_j}\cap(a,a+r)\Big )\setminus\widetilde{H}\Big |}{r} \\
&\underset{\text{by \eqref{nagy_resz} }}{<}  \frac{\frac{\eps}{2}\cdot r+\sum\limits_{j=1}^{j_r} |\left(I_{x_j}\cap(a,a+r)\right)\setminus\widetilde{H}|}{r}  
\underset{\text{by \eqref{dense2} }}{\le}  \frac{\frac{\eps}{2}\cdot r+\sum\limits_{j=1}^{j_r} \eps_{x_j}|I_{x_j}\cap(a,a+r)|}{r} \\
&\underset{\text{ by \eqref{*epskk}}}{\le} \frac{\frac{\eps}{2}\cdot r+\sum\limits_{j=1}^{j_r} \frac{\eps}{4}|I_{x_j}\cap(a,a+r)|}{r}  \underset{\text{by \eqref{szoveg_lemma} }}{\le} \frac{\frac{\eps}{2}\cdot r+\frac{\eps}{4}\cdot 2r}{r} = \eps.
\end{align*}
Similarly, we obtain
\begin{align*}
\frac{|(b-r,b)\setminus\widetilde{H}|}{r} < \eps,
\end{align*}
hence $H$ satisfies \eqref{H_jo}.

To show that $a$ is a right density point of $\widetilde{H}$ take an arbitrary $\eps^*>0$.
If $a$ is a left endpoint of $I_{x_k}$ for some $k\in\N$, we are done.
Otherwise, take an $i^*\in\N$ such that
\begin{equation}\label{i_nagy}
\frac{\eps}{i^*} < \eps^*,
\end{equation}
and a $\delta^*\in (0,b-a)$ such that
\begin{equation}\label{delta^*_kicsi}
(a,a+\delta^*)\cap \Union_{j=1}^{m_{i^*}} I_j = \emptyset.
\end{equation}
According to \eqref{H} we have that $(a,a+{\delta^*})\subset \Union_{j=1}^{\infty} I_{x_j}$, hence \eqref{delta^*_kicsi} implies that $(a,a+{\delta^*}) \subset \Union_{j=m_{i^*}}^{\infty} I_{x_j}$.
Consequently, there is a $j_{\delta^*}\in\N$ for which
\begin{equation}\label{nagy_resz_2}
\Big|(a,a+{\delta^*})\setminus\Union_{j=m_{i^*}}^{j_\delta^*} I_{x_j}\Big| < \frac{\eps^*}{2}\cdot \delta^*.
\end{equation}

Thus
\begin{align*}
&\frac{|(a,a+\delta^*)\setminus\widetilde{H}|}{\delta^*} < \frac{\Big |(a,a+\delta^*)\setminus \Union\limits_{j=m_{i^*}}^{j_{\delta^*}} I_{x_j}\Big |+\Big |\Big (\Union\limits_{j=m_{i^*}}^{j_{\delta^*}} I_{x_j}\cap(a,a+\delta^*)\Big )\setminus\widetilde{H}\Big |}{{\delta^*}} \displaybreak[0]\\
&\underset{\text{ by \eqref{nagy_resz_2}}}{\le} \frac{\frac{\eps^*}{2}\cdot \delta^*+\sum\limits_{j=m_{i^*}}^{j_{\delta^*}} |(I_{x_j}\cap(a,a+{\delta^*}))
\setminus\widetilde{H}|}{{\delta^*}} \displaybreak[0]\\
&\underset{\text{by \eqref{dense2} }}{\le} \frac{\frac{\eps^*}{2}\cdot \delta^*+\sum\limits_{j=m_{i^*}}^{j_{\delta^*}} \eps_{x_j}|I_{x_j}\cap(a,a+\delta^*)|}{\delta^*} \displaybreak[0]\\
&\underset{\text{by \eqref{eps'_kicsi} }}{\le} \frac{\frac{\eps^*}{2}\cdot \delta^*+\sum\limits_{j=m_{i^*}}^{j_{\delta^*}} \frac{\eps}{4i^*}|I_{x_j}\cap(a,a+\delta^*)|}{\delta^*} \displaybreak[0]\\
&\underset{\text{by \eqref{i_nagy} }}{\le} \frac{\frac{\eps^*}{2}\cdot \delta^*+\sum\limits_{j=m_{i^*}}^{j_{\delta^*}} \frac{\eps^*}{4}|I_{x_j}\cap(a,a+\delta^*)|}{\delta^*} \underset{\text{by \eqref{szoveg_lemma} }}{\le} \frac{\frac{\eps^*}{2}\cdot \delta^*+\frac{\eps^*}{4}\cdot 2\delta^*}{\delta^*} = \eps^*.
\end{align*}
Hence $a$ is a right density point of $\widetilde{H}$, and we obtain in the same way that $b$ is a left density point of $\widetilde{H}$.
This concludes the proof.
\end{proof}

The next lemma is in \cite[Lemma 2.4]{[BHMVlip]}.

\begin{lemma}\label{novekedes}
 Suppose that $E\subset {\mathbb R}$ and $f\colon {\ensuremath {\mathbb R}}\rightarrow {\ensuremath {\mathbb R}}$ such that $\Lip f= \mathbf{1}_E$. Then $f$ is a Lipschitz function and $|f(x)-f(y)|\le |[x,y]\cap E|$ for every $x,y\in {\ensuremath {\mathbb R}}$ (where $x<y$).
\end{lemma}

We turn now to the proof of Theorem \ref{*EEtappr}

We first note here that if there exists a $G_{\delta}$ set $\widetilde{E}$ having UDT and satisfying $|E\triangle \widetilde{E}| = 0$, then Theorem \ref{*EEtappr} trivially follows from Theorem \ref{*thUDTLip1}. However, as Theorem \ref{*thSUDTnonUDT} highlights, it is not always possible to find such a set, even if $E$ has nice density behaviour.

\begin{proof}[Proof of Theorem \ref{*EEtappr}]
The construction is analogous to, but more complicated than the proof of Theorem \ref{*thUDTLip1}, which is presented in \cite[Theorem 5.5]{[BHMVlip]}.

To avoid some technical difficulties we observe that we can suppose that we work with essentially unbounded sets, that is
for all $\aaa\in \R$ we have
 $ |E\cap (-\oo,\aaa)|>0 $
and $ |E\cap (\aaa,+\oo)|>0$.

Indeed, suppose that we proved our theorem for such cases and, for example,  there exists $\aaa\in \R$
such that 
$ |E\cap (-\oo,\aaa)|=0 $ but 
$ |E\cap (\bbb,+\oo)|>0$ for all $\bbb\in \R$.

Then one can use $E\cup (-\oo,\aaa-2]$
to obtain a $\Lip 1$ set $E'$  such that $|E'\Delta (E\cup (-\oo, \aaa-2))|=0$.

Suppose that $h$ is a continuous function  such that $\Lip h(x)=\mathbf{1}_{E'}(x)$.
Then $h'(x)=0$ on $(\aaa-2,\aaa)$, hence $E'\cap (\alpha-2,\alpha)=\emptyset$. 
Set $\widetilde{E}=(\aaa-1,+\oo)\cap E'$.

Letting 
$$
f(x)=
\begin{cases}
h(x) & \text{if } x\geq \aaa-1 \\
h(\aaa-1)  & \text{if } x< \aaa-1
\end{cases}
$$
we obtain a continuous function for which $\Lip f(x)=\mathbf{1}_{\widetilde{E}}(x)$ and $|\widetilde{E}\Delta E|=0$. 

The reduction of the other essentially bounded cases
to the unbounded case is similar.

Given an open set  $G$ we say that a set $D$ is locally finite in $G$  if
$D\sse G$ and 
for any $x\in G$ there is a $\delta>0$ such that $D\cap (x-\delta,x+\delta)$ is finite. 

We will define a nested sequence of open sets $(G_n)_{n={0 }}^\infty$ and uniformly convergent sequences of continuous functions $(f_n)_{n={0 }}^\infty$, $(\mathcal{E}_n)_{n={0 }}^\infty$ and $(\mathcal{E}^n)_{n={0 }}^\infty$ such that for every $m,n\in\N\cup \{ 0 \}$, $m\le n$  we have
\begin{enumerate}[(A)]
\item\label{mm} $|E\setminus G_n| = 0$,
\item\label{mm2} $|E\triangle \bigcap_{n={1 }}^{\infty} G_n | = 0$,
\item\label{meredek} $\Lip(f_n) \le 1$ on $\R$ for $n\geq 0$ and hence $f_{n}$ is continuous,
\item\label{E_n flat} $\mathcal{E}_n$ and $\mathcal{E}^n$ have vanishing derivative on $F_n:=\R\setminus G_n$, and $\mathcal{E}_{n}|_{F_n}=\mathcal{E}^{n}|_{F_n}=f_n|_{F_n}$,
\item\label{surun_meredek} for $n\geq 1$ there is a  locally finite set $D_n$ in $G_n$ such that for every $x\in G_n$ there are $d_1,d_2\in D_n$ for which
$$
x\in[d_1,d_2], \text{ } 0<|d_1-d_2|\le \frac{1}{n} \text{ and } \Big |\frac{f_n(d_1)-f_n(d_2)}{d_1-d_2}\Big | \ge 1-\frac{1}{n},
$$
\item\label{fix_on_D_n} $f_{n}|_{F_m\cup D_m}=f_m|_{F_m\cup D_m}$,
\item\label{boritek}$\mathcal{E}_m \le \mathcal{E}_n \le f_n \le \mathcal{E}^n \le \mathcal{E}^m$.
\end{enumerate}

Next we show that the above assumptions imply Theorem \ref{*EEtappr}.

By \eqref{surun_meredek} for every $x\in\R$ there is an $x'\in F_m\cup D_m$ such that $|x-x'| \le \frac{1}{2m}$.  Thus Lemma \ref{novekedes}, \eqref{meredek} and \eqref{fix_on_D_n} imply that
\begin{align*}
|f_n(x)-f_m(x)| &\le |f_n(x)-f_n(x')|+|f_n(x')-f_m(x')|+|f_m(x')-f_m(x)| \\
&\le |x-x'|+0+|x-x'| \le 2|x-x'| \le \frac{1}{m},
\end{align*}
that is $\|f_n-f_m\|\le \frac{1}{m}$,  i.e. the sequence $(f_n)$ is Cauchy and
therefore convergent. Let $f = \lim_{n\to\infty} f_n$.
Moreover, \eqref{fix_on_D_n} and \eqref{surun_meredek} imply that $\Lip (f) \ge 1$ on $\bigcap_{n={1 }}^{\infty} G_n$, and $\Lip (f) \le 1$ by \eqref{meredek}. According to \eqref{E_n flat}, \eqref{fix_on_D_n} and \eqref{boritek} if $x\in\R\setminus \bigcap_{n={1 }}^{\infty} G_n$ then $\Lip f(x) =0$.
Thus $\widetilde{E} := \bigcap_{n={1 }}^{\infty} G_n$ will be a suitable choice by \eqref{mm2}
and this proves Theorem \ref{*EEtappr}.

Now we turn to the proof of the fact that conditions (\ref{mm}-\ref{boritek}) can be satisfied (the places where the individual conditions are verified are marked by
$\circledast$).
Let
\begin{itemize}
\item $G_{0 } := \R$,
\item $f_{0 } :\equiv 0$,
\item $D_{0 } := \{ z\in \Z : |[z,z+1]\cap{E}| > 0 \}$.
\end{itemize}
By these definitions we can define continuous functions $\mathcal{E}_{0 }$ and $\mathcal{E}^{0 }$ for which $\mathcal{E}_{0}\le f_0 \le \mathcal{E}^{0}$ and 
\begin{equation*}
\begin{gathered}
\text{if $d_1,d_2\in D_{0}$ are adjacent in $D_{0}$ and $x\in[d_1,d_2]$ then} \\
\min\big\{f_{0}(d_1)-\mathcal{E}_{0}(x),f_{0}(d_2)-\mathcal{E}_{0}(x),  \\
 \mathcal{E}^{0}(x)-f_{0}(d_1),\mathcal{E}^{0}(x)-f_{0}(d_2)\big\} \ge |d_1-d_2|.
\end{gathered}
\end{equation*}

Now we assume that $n\in\N$, and we have already defined $G_{k}$, $f_{k}$, $\mathcal{E}_{k}$, $\mathcal{E}^{k}$ and $D_{k}$ for every $k\in \{{0 },\ldots,n-1\}$ so that they satisfy \eqref{mm}, \eqref{E_n flat}, \eqref{fix_on_D_n}, \eqref{boritek} and the following assumptions:
\begin{equation}\label{szoveg2}
\begin{gathered}
\text{if $d_1,d_2\in D_{n-1}$ are adjacent in $D_{n-1}$ and $x\in[d_1,d_2]
\sse G_{n-1}$ then} \\
\min\big\{f_{n-1}(d_1)-\mathcal{E}_{n-1}(x),f_{n-1}(d_2)-\mathcal{E}_{n-1}(x),  \\
 \mathcal{E}^{n-1}(x)-f_{n-1}(d_1),\mathcal{E}^{n-1}(x)-f_{n-1}(d_2)\big\} \ge |d_1-d_2|,
\end{gathered}
\end{equation}
if $d_1,d_2\in D_{n-1}$ are adjacent, then
\begin{equation}\label{f_lapos}
\begin{gathered}
|E\cap (d_1,d_2)| > |f_{n-1}(d_1)-f_{n-1}(d_2)|,
\end{gathered}
\end{equation}
if $n>{1 }$, and $(a,b)$ is a component of $G_{n-1}$, then
\begin{equation}\label{szoveg3}
\begin{gathered}
\left\{\text{accumulation points of $\big(D_{n-1} \cap (a,b)\big)$}\right\} = \{a,b\}.
\end{gathered}
\end{equation}

Observe that $G_{0 }$, $D_{0 }$, $f_{0 }$, $\mathcal{E}_0$ and $\mathcal{E}^0$ indeed satisfy \eqref{mm}, \eqref{boritek}, \eqref{szoveg2} and \eqref{f_lapos}. As \eqref{E_n flat} and \eqref{fix_on_D_n} say nothing when $n=0$, they also hold.

We continue by defining $G_n$.
First we define the sets $\widetilde{G}^l_n\supset {G}^l_n\supset \widetilde{G}^{l+1}_n...$ by mathematical induction.
Let
\begin{equation}\label{G^1_n}
\begin{gathered}
\begin{aligned}
\widetilde{G}^1_n := G_{n-1}
\setminus \Big (D_{n-1}\cup \Big  \{\frac{z}{n} : z\in\Z\Big  \} \cup &\Big \{\text{midpoints of the components }\\
&\text{ of } G_{n-1} \setminus D_{n-1}\Big \}\Big ).
\end{aligned}
\end{gathered}
\end{equation}

Let $l>0$ and suppose that we have already defined an open set $\widetilde{G}^{l}_n$.
According to Lemma \ref{suru} there is an open set $G^l_n\subset \widetilde{G}^l_n$ such that 
\begin{equation}\label{*ETGnl}
\text{$|(E\cap\widetilde{G}^l_n) \setminus G^l_n|=0$,
}
\end{equation}
and it also satisfies the property that if $I=(a,b)$ is a component of $G^l_n$, then $a$ is a right density point of $E\cap \widetilde{G}^l_n$, $b$ is a left density point of $E\cap \widetilde{G}^l_n$ and for every $r\in(0,b-a)$ we have
\begin{equation}\label{suru_int}
\begin{gathered}
\begin{aligned}
& \max\Big \{\frac{|(a,a+r)\setminus{E}|}{r}, \frac{|(b-r,b)\setminus{E}|}{r}\Big \} \\
&\le \max\Big \{\frac{|(a,a+r)\setminus ({E}\cap\widetilde{G}^l_n)|}{r}, \frac{|(b-r,b)\setminus({E}\cap\widetilde{G}^l_n)|}{r}\Big \} < \frac{1}{4(n+l)^2}.
\end{aligned}
\end{gathered}
\end{equation}
If $G^l_n=\emptyset$, let $I^l_j:=\emptyset$ for every $j\in\N$.
Otherwise, we take some components $I^l_1,I^l_2,I^l_3,\ldots$ of $G^l_n$ such that every bounded interval contains finitely many of them and
\begin{equation}\label{G^l_n kicsi}
\Big |G^l_n \setminus \Union_{k=1}^{\infty} I^l_k \Big | < 2^{-l}.
\end{equation}
Define $\widetilde{G}^{l+1}_n := G^{l}_n \setminus \Union_{k=1}^{\infty} I^l_k$ and continue the induction.

Set
$$
G_n := \Union_{l=1}^\infty \Union_{k=1}^\infty I^l_k.
$$

By mathematical induction for every $l^*\in\N$ we will prove 
\begin{equation}\label{kozel}
\begin{split}
\Big |E \setminus \Big(G^{l^*}_n\cup \Union_{l=1}^{l^*-1} \Union_{k=1}^\infty I^l_k\Big) \Big | = 0.
\end{split}
\end{equation} 
 As 
$$
\Big |E \setminus G^{1}_n \Big | 
\underset{\text{by \eqref{*ETGnl} }}{=} \Big |E \setminus \widetilde{G}^{1}_n \Big | \underset{\text{by \eqref{G^1_n} }}{=}\Big | E \setminus G_{n-1}\Big |,
$$
\eqref{kozel} is true for $l^*=1$. Suppose that it holds for some $l^*\in\N$, then
\begin{equation*}
\begin{split}
\Big |E \setminus \Big(G^{l^*+1}_n\cup \Union_{l=1}^{l^*} \Union_{k=1}^\infty I^l_k\Big) \Big | 
&\underset{\text{by \eqref{*ETGnl} }}{=} \Big |E \setminus \Big(\widetilde{G}^{l^*+1}_n\cup \Union_{l=1}^{l^*} \Union_{k=1}^\infty I^l_k\Big) \Big | \\
&\leq  \Big |E \setminus \Big(G^{l^*}_n\cup \Union_{l=1}^{l^*-1} \Union_{k=1}^\infty I^l_k\Big) \Big | = 0.
\end{split}
\end{equation*} 
Hence, by \eqref{G^l_n kicsi} we obtain $|E \setminus G_n|=0$.
$\circledast$ Thus (\ref{mm}) holds at step $n$.

Moreover, according to \eqref{suru_int}, if  $I\subset G_n$ is a bounded interval such that at least one of its endpoints is an endpoint of a component of $G_n$, we have that $|\big(G_n\cap I\big)\setminus{E}| \le \frac{1}{4n^2}|I|$,
$\circledast$ {which implies (\ref{mm2}).}

Now we construct $f_n$.
We set $f_n := \mathcal{E}_n := \mathcal{E}^n := f_{n-1}$ on $F_{n-1}\cup D_{n-1}$.
As \eqref{fix_on_D_n} held in the previous steps of the induction,
$\circledast$ \text{(\ref{fix_on_D_n}) holds at step $n$ as well.}

Take an arbitrary interval
\begin{equation}\label{fix_on_D_n_holds}
\text{$I = (a,b)$ contiguous to $F_{n-1}\cup D_{n-1}$.}
\end{equation}
Then $F_{0}=\ess$ and  \eqref{szoveg3} for $n>1$ imply that $a,b\in D_{n-1}$.
According to \eqref{f_lapos}, for some $k^*\in\N$ there are 
finitely many different components $I_1,\ldots,I_{k^*}$ of $G_n\cap I$ such that $|f_{n-1}(a)-f_{n-1}(b)| < \sum_{i=1}^{k^*} |I_i\cap{E}|$.
We index these components in an increasing order on the real line.
We can assume without loss of generality that $f_{n-1}(a) \le f_{n-1}(b)$.
Denote by $a_i$ and $b_i$ the endpoints of $I_i$ for every $i\in \{1,\ldots,k^*\}$, and let
$$
f_n(a_i) := f_n(a) + \frac{\sum_{j=1}^{i-1} |I_j\cap{E}|}{\sum_{j=1}^{k^*} |I_j\cap{E}|}\left(f_n(b)-f_n(a)\right)
$$
and
$$
f_n(b_i) := f_n(a) + \frac{\sum_{j=1}^{i} |I_j\cap{E}|}{\sum_{j=1}^{k^*} |I_j\cap{E}|}\left(f_n(b)-f_n(a)\right).
$$
On $I\setminus G_n$ set
\begin{equation}\label{on_compl}
\begin{gathered}
f_n(x) := \mathcal{E}_n(x) := \mathcal{E}^n(x) := \max \big(\{f_n(b_i) : i\in \{1,\ldots,k^*\} \text{ and } b_i\le x\} \cup \{f(a)\}\big).
\end{gathered}
\end{equation}

\begin{figure}[!htb]
\begin{center}
\includegraphics[width=1 \textwidth, clip]{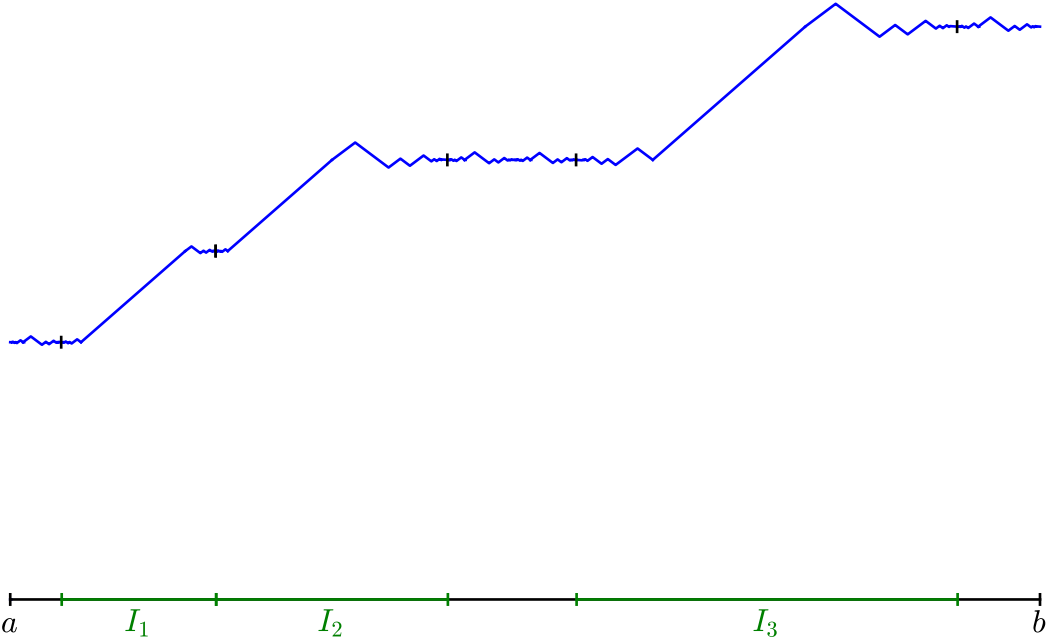}
\vspace*{-1cm}
\caption{Graph of $f_{n}$ on $I=[a,b]$}\label{ab_fig}
\end{center}
\end{figure}

Let $I'=(a',b')$ be a component of $I\cap G_n$.
We consider two cases:

\begin{enumerate}[(a)]
\item\label{keves}  Let $I'\in \{I_1,\ldots,I_{k^*}\}$.
As $a'$ is a right density point of $E$, if we choose an $a'_{0}\in I'$ close enough to $a'$, then by \eqref{suru_int} there is a $b'_{0}\in (a'_{0},b')$ such that
$$
\Big  (1-\frac{1}{n} \Big )(b'_{0}-a'_{0}) < f_n( b')-f_n( a') < |(a'_{0},b'_{0})\cap{E}|.
$$
\begin{equation}\label{*deffnaov}
\text{Set $f_n(a'_{0}) := f_n(a')$ and $f_n(b'_{0}) := f_n(b')$ and let $f_n$ be linear on $\left[a'_{0},b'_{0}\right]$.}
\end{equation}

Define $D_n$ on $[a'_{0},b'_{0}]$ such that $D_n\cap [a'_{0},b'_{0}] := \left\{a'_{0},b'_{0}\right\}$.
We have that
\begin{equation}\label{kozepen}
\text{$\circledast$ (\ref{meredek}) holds on $(a'_{0},b'_{0})$, and $\circledast$ (\ref{surun_meredek}) and (\ref{f_lapos}) hold on $[a'_{0},b'_{0}]$.}
\end{equation}

\item\label{sok} If $I'=(a',b')$ is a component of $I\cap G_n\setminus \{I_1,\ldots,I_{k^*}\}$, then set
$$
f_n(a') := f_n(b') := f_n(\max(\{a\} \cup \left\{b_i | b_i\le a'\right\})).
$$
\end{enumerate}

In the following, we will define $f_n$, $\mathcal{E}_n$, $\mathcal{E}^n$ and $D_n$ in an arbitrary component $I'$ of $I\cap G_n$.
If we do not mention which case we investigate, the statements will hold in both cases \eqref{keves} and \eqref{sok}. However, if $I'\notin\{I_1,\ldots,I_{k^*}\}$, then we put $a'_0 := b'_0 := \frac{b'+a'}{2}$ and $f_{n}(a_{0}'):=f_{n}(b_{0}'):=f_{n}(a')=f_{n}(b')$.

Let $l':= \max\{l | I'\subset \widetilde{G}^l_n\}$.
We will define a strictly decreasing sequence $(a'_{k})_{k=1}^\infty$ in $(a',a'_{0}]$ converging to $a'$.
Suppose that we have already defined $a'_0,\ldots,a'_{k-1}$ for some $k\in\N$.
We choose $a'_k\in (a',a'_{k-1})$ to satisfy
\begin{equation}\label{lapos}
\begin{gathered}
|(a'_k,a'_{k-1})| 
=  \min\Big \{\frac{1}{n+l'}|(a',a'_{k-1})|,\frac{1}{k}|(a',a'_{k-1})|+4(n+l')|(a',a'_{k-1})\setminus E|\Big \}.
\end{gathered}
\end{equation}


Next we show that 
$\lim_{k\to\infty} a'_k = a'$. Since $a'_k$ is monotone decreasing and bounded by $a'$ from below it has a finite limit $a''$. If $a'=a''$ then we are done. If $a''>a'$ then for large enough $k$ \eqref{lapos} implies that
$|(a'_k,a'_{k-1})| \ge |(a',a'_{k-1})|/k \ge |(a',a'')|/k$. Since $\sum \frac{1}{k}$ diverges, this is impossible.

By \eqref{suru_int} we have that
\begin{equation*} 
\begin{split}
4(n+l')|(a',a'_{k-1})\setminus E| 
< \frac{1}{n+l'}|(a',a'_{k-1})|,
\end{split}
\end{equation*}
hence using the fact that $4(n+l')|(a',a'_{k-1})\setminus E| $ is less than the second 
expression in $\min\{ \ , \  \}$ of \eqref{lapos} we obtain that
\begin{equation*} 
\begin{gathered}
4(n+l')|(a',a'_{k-1})\setminus E| 
 < |(a'_k,a'_{k-1})|.
\end{gathered}
\end{equation*}
This implies that
\begin{equation}\label{meredek_2}
\begin{split}
|(a'_{k},a'_{k-1})\setminus E| &\le |(a',a'_{k-1})\setminus E| < \frac{1}{4(n+l')}|(a'_{k},a'_{k-1})|.
\end{split}
\end{equation}

As $a'$ has been defined to be a right density point of $E$, by \eqref{lapos} we have
\begin{equation}\label{nulla}
\lim_{k\to\infty} \dfrac{|(a'_k,a'_{k-1})|}{|(a',a'_{k-1})|} \le \lim_{k\to\infty} \Big (\frac{1}{k}+4(n+l')\frac{|(a',a'_{k-1})\setminus {E}|}{|(a',a'_{k-1})|}\Big ) = 0.
\end{equation}

We define a sequence $(b'_k)_{k=1}^\infty$ in $(b_0,b')$ similarly.

\begin{figure}[!htb]
\begin{center}
\includegraphics[width=1 \textwidth, clip]{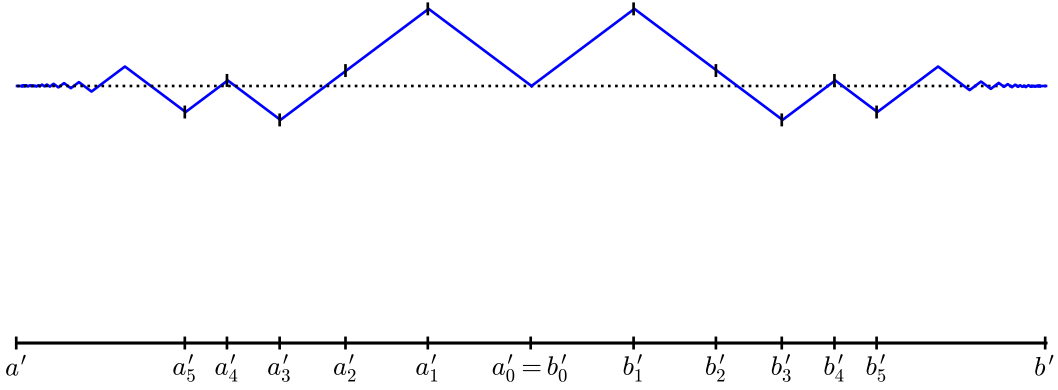}
\vspace*{-1cm}
\caption{The graph of $f_{n}$ on $I'=[a',b']$ if $I'\notin\{I_1,\ldots,I_{k^*}\}$}\label{cc_fig}
\end{center}
\end{figure}

For every $k\in\N$ let
\begin{equation}\label{cc}
f_n(a'_{k}) :=
\begin{cases}
f_n(a'_{k-1})+(1-\frac{1}{n})|(a'_{k-1},a'_{k})| \text{ if } f_n(a'_{k-1}) \le f_n(a')=f_{n}(a_{0}'), \\
f_n(a'_{k-1})-(1-\frac{1}{n})|(a'_{k-1},a'_{k})| \text{ if } f_n(a'_{k-1}) > f_n(a')=f_{n}(a_{0}'),
\end{cases}
\end{equation}
and let $f_n$ be linear on $[a'_{k},a'_{k-1}]$.
We define $f_n$ in an analogous way on $(b'_{0},b')$ using $(b'_{k})_{k=1}^\infty$  in place of $(a'_{k})_{k=1}^\infty$.

\begin{figure}[!htb]
\begin{center}
\includegraphics[width=1 \textwidth, clip]{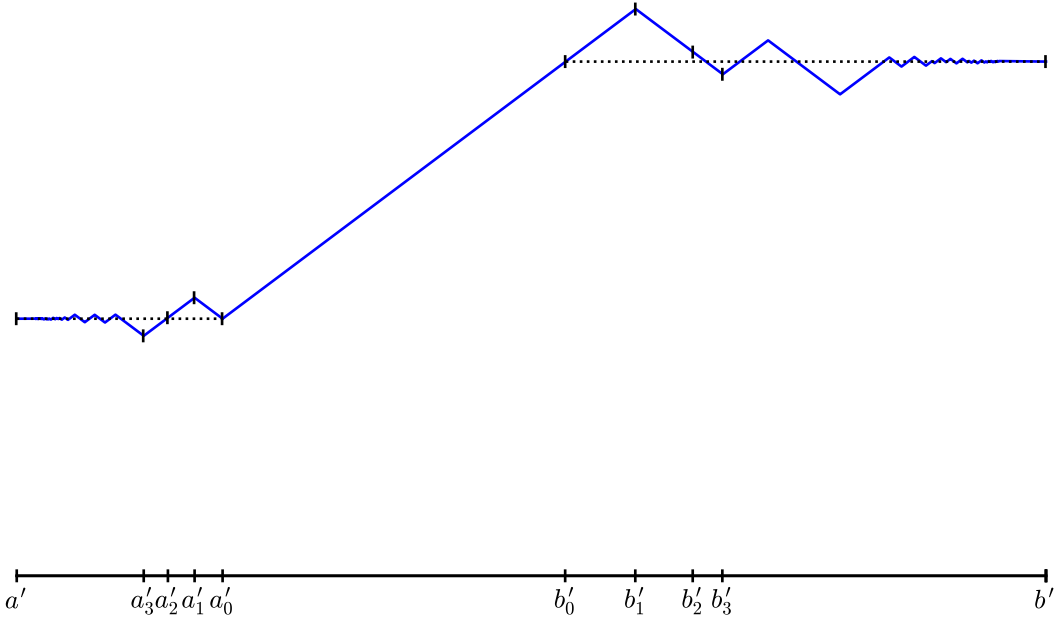}
\vspace*{-1cm}
\caption{The graph of $f_{n}$ on $I'=[a',b']$ if $I'\in\{I_1,\ldots,I_{k^*}\}$}\label{cc_ugras_fig}
\end{center}
\end{figure}

From definition \eqref{cc} and $1-\frac{1}{1}=0$ it follows that
\begin{equation}\label{konstans}
\begin{split} 
\text{if $n=1$, then $f_n|_{[a',a'_0]}\equiv f_n(a')=f_n(a'_0)$ and $f_n|_{[b'_0,b']}\equiv f_n(b'_0)=f_n(b')$.} 
\end{split}
\end{equation}
Suppose that $n>1$. By \eqref{lapos},
\begin{equation}\label{*esti3}
(a_{k-1}'-a_{k}') \le  \frac{1}{n+l'}(a_{k-1}'-a') \le \frac{1}{3}(a_{k-1}'-a').
\end{equation}

Next we show  that for all $k=0,1,...$
\begin{equation}\label{*estim}
|f_{n}(a_{k}')-f_{n}(a')|<\Big(1-\frac{1}{n}\Big)(a_{k}'-a').
\end{equation}

Observe that $0=|f_{n}(a_{0}')-f_{n}(a')|<(1-\frac{1}{n})(a_{0}'-a')$
and hence \eqref{*estim} holds for $k=0$.

Suppose that for a $k\geq 0$ we have \eqref{*estim}.

If $(f_{n}(a_{k+1}')-f_n(a'))\cdot (f_{n}(a_{k}')-f_n(a')) > 0$
then our definition in \eqref{cc} implies that \eqref{*estim}
holds for $k+1$ instead of $k$.

If $(f_{n}(a_{k+1}')-f_n(a'))\cdot (f_{n}(a_{k}')-f_n(a')) \leq 0$ then
\begin{equation*}
|f_{n}(a_{k+1}')-f_{n}(a')|
\leq |f_{n}(a_{k+1}')-f_{n}(a_{k}')|
\underset{\text{by \eqref{cc} }}{=} \Big(1-\frac{1}{n}\Big) |[a'_{k+1},a'_k]| 
\end{equation*}
$$
\underset{\text{by \eqref{*esti3} }}{\le} \Big(1-\frac{1}{n}\Big)\frac{1}{3}(a_{k}'-a')
\underset{\text{by \eqref{*esti3} }}{\le} 
\Big(1-\frac{1}{n}\Big)\cdot\frac{1}{3}\cdot \frac{3}{2} (a_{k+1}'-a').$$
Therefore, by using \eqref{cc} and\eqref{*estim} one can see that
\begin{equation}\label{changes sign}
\begin{split}
\text{ if $n>1$, then $\big(f_{n}(a_{k}')-f_n(a')\big)_{k=1}^\infty$ changes its sign infinitely often} \\
\text{ and similarly $\big(f_{n}(b_{k}')-f_n(b')\big)_{k=1}^\infty$ changes its sign infinitely often}.
\end{split}
\end{equation}
It also follows from \eqref{cc} that if $x$ is a local extremum of $f_{n}$
in $(a',a_{0}')$, then there exists $k_{x}>0$ such that $x=a_{k_{x}}$ and
\begin{equation}\label{*eqchsign}
(f_{n}(a_{k_{x}-1}')-f_n(a'))\cdot (f_{n}(a_{k_{x}}')-f_n(a')) \le 0.
\end{equation}
However, it may happen for some $k\in\N$ that $(f_{n}(a_{k-1}')-f_n(a'))$ and $(f_{n}(a'_k)-f_n(a'))$ are of the same sign.

Set
$$
D_n\cap I' := \{a'_0,a'_1,\ldots\} \cup \{b'_0,b'_1,\ldots\}.
$$
This definition means that
\begin{equation}\label{szoveg3_holds}
\text{$D_n$ satisfies (\ref{szoveg3}) on $I'$.}
\end{equation}
By \eqref{G^1_n} we have that $|I'|\le\frac{1}{n}$, hence $D_n\cap I'$ is a $\frac{1}{n}$-mesh on $I'$.
Thus by \eqref{kozepen} and \eqref{cc}, $\circledast$
\text{(\ref{surun_meredek}) is true on $I'$.}

\begin{figure}[!htb]
\begin{center}
\includegraphics[width=1 \textwidth, clip]{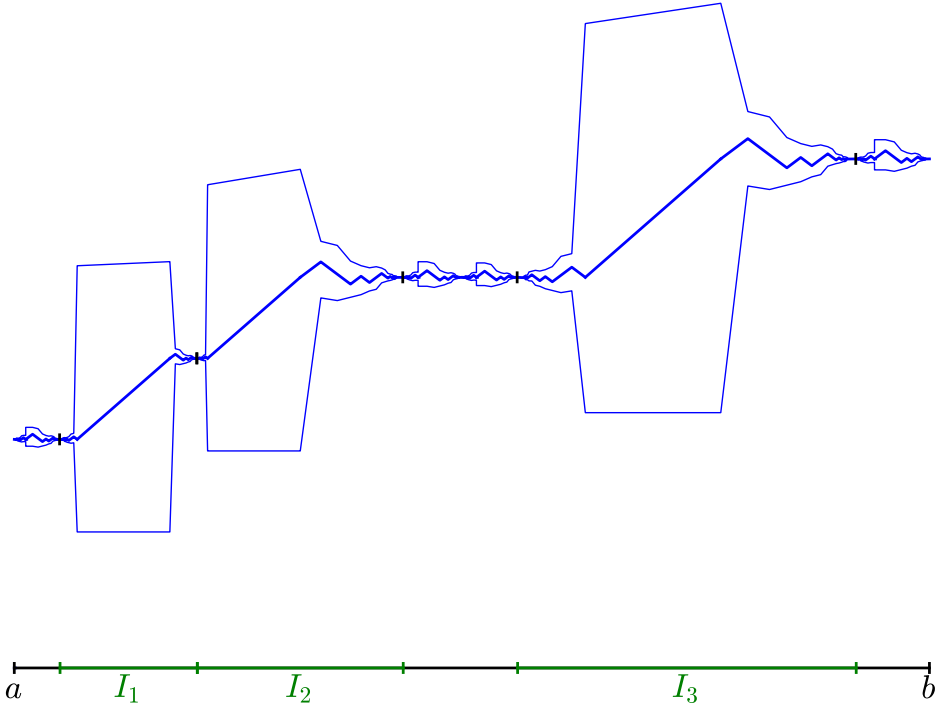}
\vspace*{-1cm}
\caption{The graphs of $\mathcal{E}_n$, $\mathcal{E}^n$ and $f_n$ on $I=[a,b]$}\label{ab_boritekkal_fig}
\end{center}
\end{figure}

By \eqref{kozepen}, \eqref{meredek_2} and \eqref{cc} we have that
\begin{equation}\label{f_lapos_holds}
\text{$f_n$ and $D_n$ satisfy (\ref{f_lapos}) on $I'$.}
\end{equation}
Moreover, \eqref{kozepen} and \eqref{cc} also imply that 
\begin{equation}\label{meredek_holds}
\text{$\circledast$ (\ref{meredek}) holds on all of $(a',b')$.}
\end{equation}
According to \eqref{konstans}, \eqref{nulla}, \eqref{cc} and \eqref{changes sign}
\begin{equation}\label{szelen lapos}
\text{the right derivative of $f_n$ is $0$ at $a'$ and the left derivative of $f_n$ is $0$ at $b'$.}
\end{equation}


If $d_1$, $d_2$ and $d_3$ are adjacent points of $D_n\cap I'$ and $d_1<d_2<d_3$, then we set
\begin{equation*}
\begin{split}
\mathcal{E}_n(d_2) := \min\left\{f_n(d_1),f_n(d_2),f_n(d_3)\right\} - \max\left\{d_2-d_1,d_3-d_2\right\}
\end{split}
\end{equation*}
and
\begin{equation*}
\begin{split}
\mathcal{E}^n(d_2) := \max\left\{f_n(d_1),f_n(d_2),f_n(d_3)\right\} + \max\left\{d_2-d_1,d_3-d_2\right\}.
\end{split}
\end{equation*}
Define $\mathcal{E}_n$ and $\mathcal{E}^n$ to be linear between adjacent points of $D_n$.
This definition immediately implies that
\begin{equation}\label{szoveg2_holds}
\text{$f_n$, $\mathcal{E}_n$, $\mathcal{E}^n$ and $D_n$ satisfy (\ref{szoveg2}) on $I'$,}
\end{equation}
and
\begin{equation}\label{f_n_kozepen}
\mathcal{E}_n \le f_n \le \mathcal{E}^n \text{ on $I=(a,b)$}.
\end{equation}
By \eqref{nulla} and \eqref{szelen lapos} we obtain that
\begin{equation}\label{szelen lapos2}
\begin{gathered}
\text{the right derivatives of $\mathcal{E}_n$ and $\mathcal{E}^n$ are $0$ at $a'$ and} \\
\text{the left derivatives of $\mathcal{E}_n$ and $\mathcal{E}^n$ are $0$ at $b'$.}
\end{gathered}
\end{equation}

Recall that $I$ was defined in \eqref{fix_on_D_n_holds}.
For every $x\in I = (a,b)$ we have that
\begin{equation*}
\begin{split}
\mathcal{E}_n(x) \ge& \min\{f_n(x') : x'\in [a,b]\} - \\
&-\max\{|d-d'| : \text{$d$ and $d'$ are adjacent elements of $D_n\cap I$}\}.
\end{split}
\end{equation*}
Hence by \eqref{meredek_holds}, \eqref{on_compl} and \eqref{G^1_n}
\begin{equation*}
\begin{split}
\mathcal{E}_n(x) \ge&\Big  (\min\{f_{n}(a),f_{n}(b)\} - \frac{b-a}{2}\Big ) - \frac{b-a}{2} \\
=& \min\{f_n(a),f_n(b)\} - (b-a) = \min\{f_{n-1}(a),f_{n-1}(b)\} - (b-a)
\end{split}
\end{equation*}
thus by \eqref{szoveg2}
$$
\mathcal{E}_n(x) \ge \min\{f_{n-1}(a),f_{n-1}(b)\} - (\min\{f_{n-1}(a),f_{n-1}(b)\} - \mathcal{E}_{n-1}(x)) = \mathcal{E}_{n-1}(x),
$$
and similarly $\mathcal{E}^n(x) \le \mathcal{E}^{n-1}(x)$.
Hence \eqref{f_n_kozepen} implies that
\begin{equation}\label{boritek_holds}
\circledast \text{ (\ref{boritek}) holds on $I$ for $n$,}
\end{equation}
since it held in the previous steps of the induction.

Take a component $I'=(a',b')$ of $I\cap G_n\setminus \{I_1,\ldots,I_{k^*}\}$.
Let $x\in (a',a'_0)$.
We want to prove that
$$
\frac{|f_n(x)-f_n(a')|}{x-a'} \le \frac{1}{n+l'-1}.
$$
We know that $f_n$ is linear between $a_k'$ and $a_{k-1}'$ for every $k\in\N$, there are infinitely many local extremum points in $\{a_0',a_1',\ldots\}$ by \eqref{*eqchsign}, and $\Lip(f_n) = 1-\frac{1}{n}$ on $(a',a'_0)$ by \eqref{cc}. 
Consequently, we can assume that $x$ is a local extremum point of $f_n$, i.e. $x=a'_{k_x}$ for some $k_x\in\N$. 
We can also suppose without loss of generality that $f_n(a'_{k_x-1})>f_n(a')$, hence 
by \eqref{*eqchsign},
$f_n(a'_{k_x})\le f_n(a')$. 
Thus 
\begin{align*}
\frac{|f_n(x)-f_n(a')|}{x-a'} &= \frac{|f_n(a'_{k_x})-f_n(a')|}{a'_{k_x}-a'} \le \frac{| f_n(a'_{k_x})-f_n(a'_{k_x-1})|}{a'_{k_x}-a'} \le \frac{a'_{k_x-1}-a'_{k_x}}{a'_{k_x}-a'} \\
&= \frac{a'_{k_x-1}-a'_{k_x}}{a'_{k_x-1}-a'-(a'_{k_x-1}-a'_{k_x})} \\
&\underset{\text{by \eqref{lapos} }}{\le} \frac{a'_{k_x-1}-a'_{k_x}}{(n+l')(a'_{k_x-1}-a'_{k_x})-(a'_{k_x-1}-a'_{k_x})} = \frac{1}{n+l'-1}.
\end{align*}
We can prove similarly for every $x\in (a',b')$ that
\begin{align}\label{f_n_lapos}
\max\Big \{\frac{|f_n(x)-f_n(a')|}{x-a'},\frac{|f_n(x)-f_n(b')|}{b'-x}\Big \}
\le \frac{1}{n+l'-1}.
\end{align}
Let $d_1,d_2,d_3,d_4\in D_n\cap I'$ be adjacent and increasing in this order.
If $x\in [d_2,d_3]$ then by the definition of $\mathcal{E}^n$ we have that 
\begin{equation}\label{E^n_lapos}
\begin{gathered}
\frac{\mathcal{E}^n(x)-\mathcal{E}^n(a')}{x-a'} \le \frac{\max\{\mathcal{E}^n(d_2),\mathcal{E}^n(d_3)\}-\mathcal{E}^n(a')}{d_2-a'} \\
\le \frac{\max\left\{f_n(d_1),f_n(d_2),f_n(d_3),f_n(d_4)\right\} + \max\left\{d_2-d_1,d_3-d_2,d_4-d_3\right\}-\mathcal{E}^n(a')}{d_2-a'}.
\end{gathered}
\end{equation} 
By \eqref{f_n_lapos}
\begin{equation*}
\begin{gathered}
\max\left\{f_n(d_1),f_n(d_2),f_n(d_3),f_n(d_4)\right\} - \mathcal{E}^n(a') \\
= \max\left\{f_n(d_1),f_n(d_2),f_n(d_3),f_n(d_4)\right\} -f_n(a') \le \frac{1}{n+l'-1}(d_4-a'),
\end{gathered}
\end{equation*}
by \eqref{lapos}
$$
\max\left\{d_2-d_1,d_3-d_2,d_4-d_3\right\} \le \frac{1}{n+l'}(d_4-a')
$$
and
$$
d_2-a' = \frac{d_2-a'}{d_3-a'} \cdot \frac{d_3-a'}{d_4-a'} \cdot (d_4-a') \ge \Big (1-\frac{1}{n+l'}\Big )^2 (d_4-a').
$$
Writing these inequalities into \eqref{E^n_lapos} we have
\begin{equation*}
\begin{split}
\frac{\mathcal{E}^n(x)-\mathcal{E}^n(a')}{x-a'} &\le \frac{\frac{1}{n+l'-1}(d_4-a')+\frac{1}{n+l'}(d_4-a')}{(1-\frac{1}{n+l'})^2 (d_4-a')} = \frac{\frac{1}{n+l'-1}+\frac{1}{n+l'}}{(1-\frac{1}{n+l'})^2 } \\
&\le \frac{\frac{1}{n+l'-1}+\frac{1}{n+l'-1}}{(1-\frac{1}{n+l'})^2} \le \frac{\frac{1}{n+l'-1}+\frac{1}{n+l'-1}}{(\frac{1}{2})^2} \le \frac{8}{n+l'-1}.
\end{split}
\end{equation*}



We can prove similarly that for every $x\in I'$
\begin{equation}\label{boritek_lapos}
\begin{split}
&\max\left\{\frac{|\mathcal{E}_n(x)-\mathcal{E}_n(a')|}{x-a'},\frac{|\mathcal{E}_n(x)-\mathcal{E}_n(b')|}{b'-x}, \right. \\
&\hphantom{\max\bigg\{}\left. \frac{|\mathcal{E}^n(x)-\mathcal{E}^n(a')|}{x-a'},\frac{|\mathcal{E}^n(x)-\mathcal{E}^n(b')|}{b'-x}\right\}
< \frac{8}{n+l'-1}.
\end{split}
\end{equation}

Next we show that the right derivative of $\mathcal{E}^n$ is $0$ on $F_n\cap [a,b)$.
Let $\eps>0$ and $x\in I\setminus G_n$.
Suppose that $x$ is not the left endpoint of a component of $G_n$ ( by \eqref{szelen lapos2}, in such endpoints $\mathcal{E}^n$ has $0$ right derivative). Then there is a positive $\delta$ such that
\begin{equation}\label{x_korny}
(x,x+\delta)\cap (I_1\cup\ldots\cup I_{k^*} \cup \Union_{l=1}^{\left\lceil 8\eps^{-1} \right\rceil} \Union_{k=1}^{\infty} I^l_n) = \emptyset.
\end{equation}
Take an arbitrary $y\in (x,x+\delta)$.
If $y\in I\setminus G_n$, then $\mathcal{E}^n(x)=\mathcal{E}^n(y)$ by \eqref{on_compl} and \eqref{x_korny}. Otherwise, we denote by $J=(a_J,b_J)$ the component of $G_n$, which contains $y$.
By \eqref{on_compl} and \eqref{x_korny}, we have that $\mathcal{E}^n(x)=\mathcal{E}^n(a_J)$, and \eqref{boritek_lapos} and \eqref{x_korny} implies
$$
\frac{|\mathcal{E}^n(y)-\mathcal{E}^n(a_J)|}{y-a_J} \le \frac{8}{n+8\left\lceil \eps^{-1} \right\rceil+1-1},
$$
hence
\begin{align*}
\frac{|\mathcal{E}^n(y)-\mathcal{E}^n(x)|}{y-x} &= \frac{|\mathcal{E}^n(y)-\mathcal{E}^n(a_J)|}{y-x} < \frac{|\mathcal{E}^n(y)-\mathcal{E}^n(a_J)|}{y-a_J} \\
&\le \frac{8}{n+8\left\lceil \eps^{-1} \right\rceil+1-1} \le \eps.
\end{align*}
It can be  verified  similarly that the left derivative of $\mathcal{E}^n$ is $0$ in $(a,b]$, and the same procedure works for $\mathcal{E}_n$.
As $I$ is an arbitrary interval contiguous to $F_{n-1}\cup D_{n-1}$, by \eqref{szoveg3} we have that $\mathcal{E}'_n = (\mathcal{E}^n)' = 0$ on $F_n\setminus F_{n-1}$.
Hence \eqref{boritek_holds} and the induction hypothesis imply that
$\circledast$ {we have proved (\ref{E_n flat}) on $I$.}


The places marked by $\circledast$ imply that all \eqref{mm}, \eqref{mm2}, \eqref{surun_meredek}, \eqref{meredek}, \eqref{E_n flat}, \eqref{fix_on_D_n} and \eqref{boritek} are satisfied for $n$, and by induction for all $n$s.
Moreover, all the assumptions \eqref{szoveg2}, \eqref{f_lapos} and \eqref{szoveg3} of the next induction step are satisfied by \eqref{szoveg2_holds}, \eqref{f_lapos_holds} and \eqref{szoveg3_holds}. This concludes the proof.
\end{proof}

\begin{acknowledgement}
We thank the referees for several comments which improved our paper.
\end{acknowledgement}


\end{document}